\documentclass[11pt]{article}

\usepackage{amsfonts}
\usepackage{amssymb,amsmath,amsthm}
\usepackage{latexsym}
\usepackage{fullpage}
\usepackage{hyperref}

\newtheorem{theorem}{Theorem}[section]

\newtheorem{lemma}[theorem]{Lemma}

\newtheorem{defn}[theorem]{Definition}

\newtheorem{claim}[theorem]{Claim}

\theoremstyle{definition}

\newcounter{tenumerate}

\def\P{\mathbb{P}}

\newcommand{\one}{\1}

\renewcommand{\epsilon}{\varepsilon}

\newcommand{\1}{\mathbf{1}}

\newcommand{\R}{{\mathbb R}}
\newcommand{\N}{{\mathbb N}}

\newcommand{\E}{{\mathbb E}}
\newcommand{\remove}[1]{}

\renewcommand{\leq}{\leqslant}
\renewcommand{\geq}{\geqslant}

\def\XXint#1#2#3{{\setbox0=\hbox{$#1{#2#3}{\int}$}
\vcenter{\hbox{$#2#3$}}\kern-.5\wd0}}

\begin{document}

\title{{\bf Scaling window for mean-field percolation of averages}}

\author{Jian Ding \footnote{Department of Mathematics,
Stanford University,
Stanford, CA 94305, USA. Email: jianding@math.stanford.edu}}

\maketitle

\begin{abstract}
For a complete graph of size $n$, assign each edge an i.i.d.\
exponential variable with mean $n$. For $\lambda>0$, consider the
length of the longest path whose average weight is at most
$\lambda$. It was shown by Aldous (1998) that the length is of order
$\log n$ for $\lambda < 1/\mathrm{e}$ and of order $n$ for $\lambda
> 1/\mathrm{e}$. Aldous (2003) posed the question on detailed
behavior at and near criticality $1/\mathrm{e}$. In particular,
Aldous asked whether there exist scaling exponents $\mu, \nu$ such
that for $\lambda$ within $1/\mathrm{e}$ of order $n^{-\mu}$, the
length for the longest path of average weight at most $\lambda$ has
order $n^\nu$.

We answer this question by showing that the critical behavior is far
richer: For $\lambda$ around $1/\mathrm{e}$ within a window of
$\alpha(\log n)^{-2}$ with a small absolute constant $\alpha>0$, the
longest path is of order $(\log n)^3$. Furthermore, for $\lambda
\geq 1/\mathrm{e} + \beta (\log n)^{-2}$ with $\beta$ a large
absolute constant, the longest path is at least of length a
polynomial in $n$. An interesting consequence of our result is the
existence of a second transition point in $1/\mathrm{e} + [\alpha
(\log n)^{-2}, \beta (\log n)^{-2}]$. In addition, we demonstrate a
smooth transition from subcritical to critical regime. Our results
were not known before even in a heuristic sense.
\end{abstract}


\section{Introduction}
In this work, we study the stochastic mean-field distance model. For
a complete graph $G = (V, E)$ of size $n$, associate each edge $e\in
E$ a non-negative weight $X_e$ which is an independent exponential
variable with mean $n$. For $\lambda > 0$, let $L(n, \lambda)$ be
the length of the longest path whose average weight is at most
$\lambda$.  It was shown by Aldous \cite{Aldous98} that with high
probability (with probability tending to 1 as $n\to \infty$) $L(n,
\lambda) = O(\log n)$ for $\lambda < 1/\mathrm{e}$ and $L(n,
\lambda) = \Theta(n)$ for $\lambda
> 1/\mathrm{e}$. Aldous \cite{Aldous03} posed the question on the detailed
behavior of $L(n, \lambda)$ at and near criticality $1/\mathrm{e}$.
In particular, Aldous asked whether there exist scaling exponents
$\mu, \nu$ such that
\begin{equation}\label{eq-aldous}
n^{-\mu} L(n,\mathrm{e}^{-1} +xn^{- \nu}) \to m(x)\end{equation} in
probability for some deterministic function $m(x)$ satisfying
$$\lim_{x \to \infty} m(x) = \infty, \lim_{x \to
-\infty} m(x) = 0\,.$$

We show in this work that the critical behavior for the stochastic
mean-field model is different from and far richer than that
questioned as in \eqref{eq-aldous}. Our first result determines the
order of $L(n, \lambda)$ at criticality as well as establishes the
right order for the critical window, as incorporated below.
\begin{theorem}\label{thm-critical}
There exist absolute constants $\alpha, C, c>0$ such that for all
$\mathrm{e}^{-1} - (\log n)^{-2} \leq \lambda \leq \mathrm{e}^{-1} +
\alpha (\log n)^{-2}$
\begin{equation*}
\P\left(c(\log n)^3 \leq L(n, \lambda) \leq C (\log n)^3\right) \to
1\,,\end{equation*}
\end{theorem}
\noindent{\bf Remark.} In a recent private communication, Aldous
made a guess that $L(n, 1/\mathrm{e}) = n^{o(1)}$, which is
confirmed by the preceding theorem.

Our second result shows a lower bound of polynomial in $n$ on $L(n,
\lambda)$ if $(\lambda - 1/\mathrm{e})/(\log n)^2$ exceeds a large
absolute constant.
\begin{theorem}\label{thm-supcritical}
There exist absolute constants $\beta, C>0$ such that for all
$\lambda \geq \mathrm{e}^{-1} + \beta (\log n)^{-2}$
\begin{equation*}
\P\left(n^{1/4} \leq L(n, \lambda) \leq C n (\lambda -
\mathrm{e}^{-1})\right) \to 1\,,\end{equation*}
\end{theorem}
\noindent{\bf Remark.} It seems more careful analysis can improve
the lower bound to $n^{1/2+o(1)}$. We chose not to do so because we
believe it is still far away from being tight and thus the
improvement is only technical.

 Interestingly,
Theorems~\ref{thm-critical} and \ref{thm-supcritical} imply that
there is yet another phase transition occurring somewhere in
$1/\mathrm{e} + [\alpha (\log n)^{-2}, \beta (\log n)^{-2}]$. In
addition, we demonstrate a smooth transition from subcritical to
critical regime.
\begin{theorem}\label{thm-subcritical}
There exist absolute constants $C, c>0$ such that for all $\lambda
\leq \mathrm{e}^{-1} - (\log n)^{-2}$,
\begin{equation*}
\P\left(c (\mathrm{e}^{-1} - \lambda)^{-1} \log n \leq L(n, \lambda)
\leq C (\mathrm{e}^{-1} - \lambda)^{-1} \log n\right) \to
1\,.\end{equation*}
\end{theorem}

\smallskip

\noindent{\bf Related work.} While our work focuses on the second
order behavior (or finite-size scaling in the language of
statistical physics), the first-order behavior was studied by Aldous
\cite{Aldous05}.  It is believed that $L(n, \lambda)/n \to
\delta(\lambda)$ in probability as $n\to \infty$ for some function
$\delta(\lambda)$. Indeed, we see that $\delta(\lambda) = 0$ for
$\lambda < 1/\mathrm{e}$. In \cite{Aldous05}, a non-rigorous
derivation of $\delta(\lambda)$ using a reformulation of the cavity
method gives that $\delta(\lambda) \asymp (\lambda -
1/\mathrm{e})^3$.

In addition, the quantity $L(n, \lambda)$ studied in this paper is a
natural variant of several other objects that were studied before.
If we consider the path of small maximal weight other than average
weight, this is an extensively studied question of the longest path
in Erd\H{os}-R\'enyi random graphs. For $G \sim G(n, c/n)$ (a random
graph obtained by preserving each edge in complete graph with
probability $c/n$ independently), Ajtai, Koml{\'o}s, and
Szemer{\'e}di \cite{AKS81} proved that there is a path of length
$\alpha(c) n$ where $\alpha(c) > 0$ for $c>1$ and $\alpha(c) \to 1$
as $c\to \infty$; a similar and slightly weaker result was shown
independently by Fernandez de la Vega \cite{Fernandez79}. Later, the
attention was shifted to the asymptotics of $1 - \alpha(c)$.
Improving a previous work of Bollob{\'a}s \cite{Bollobas82}, Frieze
\cite{Frieze86} obtained a sharp estimate on the asymptotics of $1 -
\alpha(c)$ as $c\to \infty$. In addition, it is not hard to see that
for $c<1$ the longest path is of order $\log n$, and for $c = 1$ it
can be deduced from a result of Nachmias and Peres \cite{NP08} that
the longest path is of order $n^{1/3}$.

Here we fix a maximum $\lambda$ for the average weight and try to
maximize the length for the longest path that satisfies this
constraint. If we reverse the optimization (i.e., we insist on a
path that visits every vertex and minimize the average weight), it
becomes the classic traveling salesman problem in the mean-field
setting. For this question, W\"{a}stlund \cite{Wastlund10}
established the sharp asymptotics for more general distributions on
the edge weight, confirming the Krauth-M\'{e}zard-Parisi conjecture
\cite{MP86b, MP86, KM89}.

Finally, the maximal size $T(n, \lambda)$ of the subtree whose
average weight is at most $\lambda$ was studied in \cite{Aldous98}.
It was shown that $T(n, \lambda)$ transitions from $o(n)$ to
$\Theta(n)$ at some critical point $\lambda_0$ whose value can be
specified in terms of a fixed point of a mapping on probability
distributions.

\noindent{\bf Remark.} While our work was in review, Mathieu and Wilson posted an article \cite{MW12} studying the minimal mean-weight cycles in the same setting, where they demonstrated a transition at $1/\mathrm{e}$.
\smallskip

\noindent{\bf Main ideas of the proofs.}  We view the problem from a
slightly different perspective. We first fix a length $\ell$ and
compute the minimal average weight of all paths of length $\ell$;
then we vary $\ell$ to match this minimal average weight with
$\lambda$. A simple and useful fact is that the minimal average
weight is increasing with $\ell$ at least in a coarse sense (c.f.
Claim~\ref{claim-path-monotone}).

With the aforementioned perspective in mind, our proof ideas can be
traced back to Bramson's celebrated work \cite{Bramson78}, which
gives a very precise evaluation on the minimal displacement of the
branching Brownian motion. The main obstacle is that, we do not have
a real tree structure in the mean-field setting (one could argue
that it is locally tree-like, but certainly not globally), which is
a crucial component in Bramson's argument. The lacking of a tree
structure poses the challenges on how to control the correlation
between different paths (say, of the same length) and how to select
the truncation function for the second moment calculation, where the
two issues are intrinsically related to each other. In what follows,
we discuss the solution to these challenges focusing on the case of
$\lambda = 1/\mathrm{e}$.

The solution arises from the following observations. Note that there
are two opposite forces on the (maximal) deviation of the partial
sums from the expectation for a typical path with small average
weight. First, the deviation cannot be too large since otherwise
there would exist a path whose average weight is too small, but that
is unlikely due to a first moment calculation (c.f.,
Lemma~\ref{lem-atypical}). Second, the deviation cannot be too small
for long paths since conditioning on the average weight of a path,
the partial sums behave like a Brownian bridge and thus typically
exhibits a deviation of order $\sqrt{\mbox{path length}}$ (c.f.,
Lemma~\ref{lem-deviation}). These two forces together, leaves no
other possibility but that the path of small average weight is
short.

Crucially, the aforementioned deviation of the path serves well as
the truncation function (for the proof of the lower bound). Observe
that a bad event which produces large probability for the average
weights of both two paths to be small, is that the weights on the
common edges for these two paths are unusually small. However, once
restricted to paths of small deviation, the total weight of the
common edges cannot differ much from the expectation (given the
average weight of the path) and it is indeed bounded by the maximal
deviation multiplied with the number of segments induced by these
common edges (c.f., Definition \ref{def-theta}). Another important
ingredient is that the number of pairs (of the paths) decreases
rapidly with the number of segments of the common edges (c.f.,
Lemma~\ref{lem-counting}). Altogether, this allows us to control the
correlation globally, and thus provides a way to prove the lower
bound.

\smallskip

\noindent{\bf Discussions and further questions.} Our work suggests
a number of open questions. Naturally, one could ask what is the
location and behavior for the second phase transition. The main
obstacle for identifying the transition location seems to be that we
have to select different truncation functions for the upper and
lower bounds in the proof. More importantly, the probability costs
for these two truncations are hugely different. The argument of
Bramson also adopts different truncations (the so-called upper
 and lower envelops), but the probability costs for these two
turn out to be of the same order in that case.

It would also be interesting to determine the right order of $L(n,
\lambda)$ in the regime of Theorem~\ref{thm-supcritical}. The lower
bound we obtained there seems to be far away from being tight. The
main limitation of our arguments is that, we rely heavily on the
fact that the number of pairs (of paths) decreases rapidly with the
number of segments for the common edges assuming a fixed number of
common edges. This stops being true once the path under
consideration gets too long.

An alternative direction is on the refined estimate at criticality.
In particular, does \begin{equation}L(n, \mathrm{e}^{-1})/(\log n)^3
\to \xi\end{equation} in probability for some $\xi>0$? If so, what
is the limit and what is the variance of $L(n, 1/\mathrm{e})$?

\smallskip

\noindent {\bf A word on notation.} Throughout the paper, we denote
by $C, c>0$ absolute constants whose value could vary from line to
line. Other absolute constants like $\alpha, \beta, C^\star,
c^\star$ are fixed once for all. As we have different regimes to
consider, we usually fix the value of the parameter $\lambda$ and
possibly other parameters in each section/subsection, and all of
these settings for values will appear at the very beginning at each
section/subsection.

\section{Critical behavior within scaling window}
\label{sec:critical}

In this section, we study the critical behavior within scaling
window and prove Theorem~\ref{thm-critical}.

\subsection{Deviation of typical light path}

For a path $\gamma = v_0, e_1, v_1, \ldots, e_\ell, v_\ell$ where
$v_{i-1}$ and $v_i$ are endpoints of $e_i$ for all $i\in [\ell]$ (of
course a sequence of edges would already uniquely specify a path,
but we purposely choose to emphasize both vertices and edges for a
path in this work), let
\begin{equation}\label{eq-def-total-weight}
X(\gamma) = \mbox{$\sum_{i=1}^\ell$} X_{e_i}
\end{equation} be
the (total) weight of $\gamma$. Clearly, $X(\gamma)$ follows Gamma
distribution, which is of central importance throughout the work.
Let $Z\sim \Gamma(\theta, k)$ be a Gamma variable with parameter
$(\theta, k)$, that is to say, $Z$ has the same law as a sum of $k$
i.i.d.\ exponential variables with mean $\theta$. We will repeatedly
use the density function $f_{\theta, k}(z)$ of $Z$, where
\begin{equation}\label{eq-gamma-distribution}
f_{\theta, k}(z) = z^{k-1} \frac{\mathrm{e}^{-z/\theta}}{\theta^k
(k-1)!} \mbox{ for all } z\geq 0, \theta>0, k\in \N\,.
\end{equation}
We first show that the average weight of a path cannot be
significantly smaller than $1/\mathrm{e}$. For convenience of
notation, denote by $\Gamma_\ell$ the collection of all paths of
length $\ell$, for any $\ell \in [n]$.
\begin{lemma}\label{lem-atypical}
Let $E_n$ be the event that there is a path of length $\ell$ with
weight at most $\ell/\mathrm{e} - \log n$ by
\begin{equation}\label{eq-def-E}
E_n = \cup_{\ell=1}^n\cup_{\gamma\in \Gamma_\ell}\{ X(\gamma) \leq
\mathrm{e}^{-1} \ell - \log n\}\,.\end{equation} Then $\P(E_n)\to 0$,
as $n \to \infty$.
\end{lemma}
\begin{proof}
For any $\ell\in [n]$, we have $|\Gamma_\ell| \leq n^{\ell+1}$. In
addition, by \eqref{eq-gamma-distribution}, the probability for each
$\gamma\in \Gamma_\ell$ has total weight less than $\mathrm{e}^{-1}
\ell - \log n$ is bounded by
$$\P\left(X(\gamma) \leq \mathrm{e}^{-1} \ell - \log n\right) \leq 10 \cdot (\mathrm{e}^{-1} \ell - \log n)^{\ell - 1} \frac{\mathrm{e}^{-\ell/n}}{n^\ell (\ell-1)!} = O(n^{-(\ell + \mathrm{e})})\,,$$
where the last equality follows from Stirling's formula. An
application of a union bound over $\Gamma_\ell$ and then over
$\ell\in [n]$ yields the lemma.
\end{proof}
Define $M(\gamma)$ to be the deviation of $\gamma$ away from the
linear interpolation between the starting and ending edges, by
\begin{equation}\label{eq-def-deviation}
 M(\gamma) =
\mbox{$\sup_{1\leq k\leq \ell}$} |\mbox{$\sum_{i=1}^k$} X_{e_i} -
\tfrac{k}{\ell} X(\gamma)|\,.\end{equation} By Donsker's theorem, it
is not hard to say that the deviation process $\{\sum_{i=1}^k X_{e_i} -
\tfrac{k}{\ell} X(\gamma)\}$ conditioned on the value of $X(\gamma)$
converges to a Brownian bridge after suitable normalization. The
deviation of a Brownian bridge, i.e., the maximum of the absolute
values, is known to have Kolmogorov distribution, where the law can
be written down explicitly as a sum of series. In particular, its
left tail area has been obtained in \cite{PG76} as follows:
\begin{equation}\label{eq-lower-tail}\P\left(\mbox{$\max_{0\leq t\leq 1}$}|B_t| \leq \delta\right) =
\frac{\sqrt{2 \pi}+o_\delta(1)}{\delta}
\mathrm{e}^{-\frac{\pi^2}{8\delta^2}}\,,\end{equation} where
$(B_t)_{0\leq t\leq 1}$ is a standard Brownian bridge, and $0\leq
o_\delta(1) \downarrow 0$ as $\delta\to 0$. The analog to deviation
of Brownian bridge gives convincing evidence for the type of decay
for the lower tail of $M(\gamma)$. However, as we are trying to
analyze the tiny probability for a rare event, the desired estimate
could not follow directly by convergence in law. We give a proof in
what follows, without aiming at optimizing the exponents for the
decay. We start with the next simple claim.

\begin{claim}\label{claim-exponential}
For i.i.d.\  exponential variables $Z_i$ with mean $\theta > 0$ and
$m\leq n/2$, let $Z = \sum_{i=1}^n Z_i$ and $Z' = \sum_{i=1}^m Z_i$.
Let $g(\cdot)$ be the density function of $Z'$ conditioned on $Z =
\theta n$. Then for all $1 \leq |z - \theta m| \leq 10 \sqrt{\theta
m}$,
$$\tfrac{1}{10^6 \sqrt{\theta m}}\leq g(z) \leq \tfrac{2}{\sqrt{\theta m}}\,.$$
\end{claim}
\begin{proof}
Let $f_{k, \theta}(\cdot)$ be density function of Gamma distribution
as in \eqref{eq-gamma-distribution}. By Bayesian formula, we obtain
that
\begin{align*}
g(z) = \frac{f_{m, \theta} (z) f_{n-m, \theta}(\theta n - z)}{f_{n,
\theta} (\theta n)} = \frac{z^{m-1} \mathrm{e}^{-z/\theta}
}{\theta^m (m-1)!}\frac{(\theta n - z)^{n-m-1} \mathrm{e}^{-(\theta
n -z)/\theta}}{\theta^{n-m} (n-m-1)!} \frac{\theta^n (n-1)!}{(\theta
n)^{n-1}\mathrm{e}^{-n}}\,.
\end{align*}
Now the claim follows from a direct computation with an application
of Stirling's formula.
\end{proof}

\begin{lemma}\label{lem-deviation}
Let $Z_i$ be i.i.d.\ exponential variables with mean $\theta>0$ for
$1\leq i\leq n$. For $1/4\leq \rho \leq 4$, consider the variable
\begin{equation}\label{eq-def-M}
M = M_n = \mbox{$\sup_{1\leq k\leq n}$}|\mbox{$\sum_{i=1}^k$} Z_i -
\rho k|.\end{equation} Then, there exist absolute constants $c^\star,
C^\star>0$ such that for all $r\geq 1$ and $n\geq r^2$,
$$\mathrm{e}^{-C^\star n/r^2}\leq \P(M\leq r \mid \mbox{$\sum_{i=1}^n$} Z_i = \rho n) \leq \mathrm{e}^{-c^\star n/r^2} \,.$$
\end{lemma}
\begin{proof}
First observe a useful property for exponential variables: for
i.i.d. exponential variables $Y_i$ with mean $\theta_1$ and i.i.d.
exponential variables $Z_i$ with mean $\theta_2$ for any $\theta_1,
\theta_2>0$, we have that for all $k\in \N$ and $z>0$
\begin{equation}\label{eq-identity-in-law}
(Y_1, Y_2, \ldots, Y_k \mid \mbox{$\sum_{i=1}^k$} Y_i = z)
\stackrel{law}{=} (Z_1, Z_2, \ldots, Z_k \mid \mbox{$\sum_{i=1}^k$}
Z_i = z)\,,
\end{equation}
since both vectors are uniformly distributed over $\{(z_1, \ldots,
z_k) \in \mathbb{R}_+^k: \sum_i z_i = z\}$ conditioning on the sums
being $z$ (Note that this property is known in Statistics as ``sufficiency'' of the sample mean for the parameter in the family of Exponential distributions index by the mean). Therefore, we can in what follows assume that $\theta =
\rho$.

We now give a proof for the upper bound. For convenience, we assume
that $r$ is a positive integer. The intuition (also for the lower
bound) is that we can divide $n$ into blocks of size $r^2$, and in
every such a block the fluctuation of the path is of order $r$ and
thus the probability for the path in this block to stay within $[-r,
r]$ is bounded away from 0 and 1. Since the number of blocks is
$n/r^2$, this gives the right type of decay for the lower tail of
fluctuation. Precisely, for $j=1, \ldots, \lfloor n/r^2 \rfloor -
1$, define the event $Q_j$ by
$$Q_j = \cap_{(j-1)r^2\leq k\leq jr^2}\{|\mbox{$\sum_{i=1}^k$} Z_i - \rho k| \leq r\}\,.$$
It is clear that
$$\P(Q_{j+1} \mid Q_1, \ldots, Q_j) \leq \P\left(|\mbox{$\sum_{(j-1)r^2\leq i\leq jr^2}$}Z_i - \rho r| \leq 2r \mid Q_1, \ldots Q_j\right) \leq 1 - 10^{-7}\,,$$
where the last step follows from Claim~\ref{claim-exponential}. This
yields that
$$\P\left(M\leq r \mid \mbox{$\sum_{i=1}^n$} Z_i = \rho n\right) \leq \P\left(Q_j: \forall 1\leq j\leq \lfloor n/r^2 \rfloor - 1\right) \leq (1 - 10^{-7})^{\lfloor n/r^2 \rfloor - 1}\,.$$

Now we turn to the proof of lower bound. For $j = 1, 2, \ldots,
\lfloor n/r^2\rfloor$, define the event
$$R_j = \{|\mbox{$\sum_{i=1}^{jr^2}$} Z_i - \rho j r^2| \leq r/2\}  , \mbox{ and } S_j = \cap_{(j-1)r^2\leq k\leq jr^2}\{|\mbox{$\sum_{i= (j-1)r^2}^k$} Z_i - \rho (k - (j-1)r^2)| \leq r/2\}.$$
It is clear from triangle inequality that
\begin{equation}\label{eq-containment}\cap_{1\leq j\leq n/r^2 - 1} R_j \cap \cap_{1\leq
j\leq n/r^2}S_j \subseteq \{M \leq r\}\,.\end{equation} Furthermore,
by Claim~\ref{claim-exponential} again, we get that (write
$\mathcal{R} = \cap_{1\leq j\leq n/r^2 - 1} R_j$)
\begin{equation}\label{eq-prob-R}\P(\mathcal{R}) \geq \mbox{$\prod_{1\leq j\leq n/r^2 -
1}$}\P(R_j \mid R_1, \ldots, R_{j-1}) \geq 10^{-8
n/r^2}\,.\end{equation} In addition, conditioned on $\sum_{i=
(j-1)r^2}^{j r^2}Z_i = s$, we see that $\frac{1}{\sqrt{r}}(\sum_{i=
(j-1)r^2}^{(j-1)r^2 + tr^2} Z_i)$ for $1\leq t\leq 1$ converges to a
standard Brownian bridge, and thus (see \eqref{eq-lower-tail}) for
$r\geq r_0$ where $r_0$ is a large absolute constant, we have that
$\P(S_j\mid \mathcal{R}, \mbox{$\sum_{i= (j-1)r^2}^{j r^2}$}Z_i = s)
\geq 10^{-2}\,.$ Trivially, for $r\leq r_0$ we have $\P(S_j\mid
\mathcal{R}, \mbox{$\sum_{i= (j-1)r^2}^{j r^2}$}Z_i = s) \geq
10^{-r_0}$. Since given the sum in each block, the variables in
different blocks are independent, we can then deduce that
$\P(\cap_{j=1}^{n/r^2} S_j\mid \mathcal{R}) \geq 10^{-r_0 n/r^2}$.
Combined with \eqref{eq-containment} and \eqref{eq-prob-R}, this
gives the desired lower bound with $C^\star = 3r_0$.
\end{proof}

\subsection{Upper bound}\label{sec:critical-upper}
In this subsection, we prove the upper bound for
Theorem~\ref{thm-critical}. Recall that $c^\star>0$ is the absolute
constant defined in Lemma~\ref{lem-deviation}. Set
\begin{equation}\label{eq-def-alpha}
\alpha = \tfrac{c^{\star}}{27\mathrm{e}}\,.\end{equation} Fix in
this subsection
$$\lambda = \lambda_\alpha = \mathrm{e}^{-1} + \alpha (\log
n)^{-2}\,.$$  By monotonicity, it suffices to give an upper bound on
$L(n, \lambda)$. We start with a simple claim, reducing the
consideration to paths of length between $[\ell, 2\ell)$ for the
purpose of showing $L(n, \lambda) \leq \ell$.
\begin{claim}\label{claim-path-monotone}
The following holds deterministically. Suppose that there exists a
path of length $L \geq \ell$ such that the average weight is $\zeta$
for some $\zeta>0$. Then there exists a path of length between
$[\ell, 2\ell)$ such that the average weight is at most $\zeta$.
\end{claim}
\begin{proof}
Let $\gamma$ be a path of length $L\geq \ell$ with
average weight $\zeta$. Suppose that $\gamma$ consists a sequence of consecutive edges $e_1, \ldots, e_L$. Write $k = \lfloor L/
\ell \rfloor$, and we divide $\gamma$ into a collection of
$k$ edge-disjoint paths where $\gamma_i$ consists of edges $e_{i\ell+1}, \ldots, e_{(i+1)\ell}$ for $0\leq i< k-1$ and $\gamma_{k-1}$ consists of edges $e_{(k-1)\ell + 1}, \ldots, e_{L}$. Obviously $|\gamma_i| = \ell$ for all $0\leq i<k-1$ and
$|\gamma_{k-1}|\in [\ell, 2\ell)$.
Since $\gamma = \cup_{i=0}^{k-1} \gamma_i$, we see that at least one of paths $\gamma_i$ must have
average weight at most $\zeta$, as required.
\end{proof}

We now show that there cannot exist a long path with small average
weight but even moderately large deviation.
\begin{lemma}\label{lem-upper-critical-1}
For $\ell = (\log n)^3/\alpha$, we have
$$\P\left(\exists \ell\leq \ell'<2\ell, \gamma\in \Gamma_{\ell'}: X(\gamma) \leq \lambda \ell', M(\gamma) \geq 3 \log n\right) \to 0\,.$$
\end{lemma}
\begin{proof}
Suppose there exists $\ell \leq \ell'<2\ell$ and $\gamma\in
\Gamma_{\ell'}$ such that $X(\gamma) \leq \lambda \ell'$ and
$M(\gamma) \geq 3 \log n$. Denote by $\gamma = v_0, e_1, v_1,
\ldots, e_{\ell'}, v_{\ell'}$, and let $\ell^*$ be such that
$$M(\gamma) = |\mbox{$\sum_{i=1}^{\ell^*}$} X_{e_i} - \tfrac{\ell^*}{\ell'} X(\gamma)|\,.$$
Consider two sub-paths $\gamma_1 = v_0, e_1, \ldots, e_{\ell^*}, v_{\ell^*}$ and
$\gamma_2 = v_{\ell^*}, e_{\ell^*+1}, \ldots, e_{\ell'}, v_{\ell'}$.
By our assumption on $\gamma$ and definition of $\ell^*$, we have
$$\mbox{either } \quad  X(\gamma_1) \leq \mathrm{e}^{-1} \ell^{*} - \log
n \quad \mbox{ or } \quad  X(\gamma_2) \leq \mathrm{e}^{-1} (\ell' -
\ell^*) - \log n\,.$$ This implies that
$$\{\exists \gamma\in \Gamma_{\ell'}: X(\gamma) \leq \lambda \ell', M(\gamma) \geq 3 \log n\} \subseteq E_n\,,$$
where $E_n$ is the event defined in \eqref{eq-def-E}. The desired
estimate now follows from Lemma~\ref{lem-atypical}.
\end{proof}

We next turn to control paths with small deviation.
\begin{lemma}\label{lem-upper-critical-2}
For $\ell =  (\log n)^3/\alpha$, we have
$$\P\left(\exists \ell \leq \ell'<2\ell, \gamma\in \Gamma_{\ell}: X(\gamma) \leq \lambda \ell, M(\gamma) \leq 3 \log n\right) \to 0\,.$$
\end{lemma}

\begin{proof}
Fix an $\ell'$ with $\ell\leq \ell'<2\ell$ and fix  $\gamma\in
\Gamma_{\ell'}$. By \eqref{eq-gamma-distribution} and
Lemma~\ref{lem-deviation}, we obtain that
$$\P\left(X(\gamma) \leq \lambda \ell', M(\gamma) \leq 3 \log n\right) =O(1)n^{-\ell'} \ell'^{-1/2} \mathrm{e}^{2\mathrm{e}\log n} \mathrm{e}^{-c^\star \log n/(9\alpha)} = O(1)n^{-\ell'} n^{- \mathrm{e}}\,,$$
where the last equality follows from the definition of $\alpha$ in
\eqref{eq-def-alpha}. Noting that $|\Gamma_\ell'| \leq n^{\ell'+1}$,
we deduce the desired result by first applying a union bound over
$\gamma\in \Gamma_{\ell'}$ and then over $\ell\leq \ell'<2\ell$.
\end{proof}

The upper bound for Theorem~\ref{thm-critical} is an immediate
consequence of Lemmas~\ref{lem-upper-critical-1} and
\ref{lem-upper-critical-2}, together with
Claim~\ref{claim-path-monotone}.

\subsection{Lower bound}\label{sec:critical-lower}
In this subsection, we prove the lower bound for
Theorem~\ref{thm-critical}. Fix $\lambda = \mathrm{e}^{-1} - (\log
n)^{-2}$ in this subsection, and it suffices to establish the lower
bound on $L(n, \lambda)$. Let $c, \delta>0$ be two small absolute
constants to be selected. For $1\leq \ell\leq n$ and $\gamma \in
\Gamma_\ell$, define
\begin{equation}\label{eq-def-F}
F_\gamma = \{\lambda \ell - 1 \leq X(\gamma)\leq \lambda \ell,
M(\gamma) \leq \delta \log n\}\,.
\end{equation}

By \eqref{eq-gamma-distribution} and Lemma~\ref{lem-deviation}, we
obtain that for all $\gamma\in \Gamma_\ell$ with $\ell = c (\log
n)^3$
\begin{equation}\label{eq-prob-F} \P(F_\gamma) \geq \tfrac{1}{100}
n^{-\ell} \ell^{-1/2} n^{-ec - cC^\star/\delta^2}\,,\end{equation}
where $C^\star$ is the absolute constant from
Lemma~\ref{lem-deviation}. Defining $$N = \mbox{$ \sum_{\gamma\in
\Gamma_\ell}$}\one_{F_\gamma}\,,$$ we see that the first moment of
$N$ would be large if we select $c, \delta$ properly. The key issue
here is to bound the second moment of $N$.

\begin{lemma}\label{lem-second-moment}
Consider $\ell = c (\log n)^3$. For any $\gamma\in \Gamma_\ell$, we
have that
$$\mbox{$\sum_{\gamma'\in \Gamma_\ell}$} \P(F_\gamma \cap F_{\gamma'}) \leq
\P(F_\gamma) \cdot (\E N + O(1)\ell^3 n^\delta)\,.$$
\end{lemma}
In order to prove the preceding lemma, one needs to study the
correlation structure between $\gamma$ and all other paths in
$\Gamma_\ell$. In order to have a global control of the correlation
between $\gamma$ and all other paths, a natural strategy is to first
select different scales of correlations and then estimate the
cardinality of the paths that fall into each scale. This strategy
was implemented in the case of branching Brownian motion where the
scale for correlation is chosen to be the number of common edges
between a path and the path $\gamma$. In the mean-field setting, we
do not really have a tree structure (as two paths can bifurcate and
merge and then bifurcate...). In addition, merely the number of
common edges does not seem to fully characterize the correlation
between two paths. Therefore, we need to choose an auxiliary
quantity which together with the number of common edges, offers an
effective measurement of the correlation. We elaborate in what
follows.

For a path $\gamma$, denote by $E(\gamma)$ the collection of edges
in $\gamma$. For $S\subseteq E(\gamma)$, we call a segment of
$\gamma$ an $S$-component if it is a maximal segment of $\gamma$
where all the edges belong to $S$.
\begin{defn}\label{def-theta} For two paths $\gamma$ and $\gamma'$, we define a functional
$\theta(\gamma, \gamma')$ to be the number of $S$-components of
$\gamma$ where $S = E(\gamma) \cap E(\gamma')$.\end{defn} The
functional $\theta(\gamma, \gamma')$ turns out to be a good
additional measurement for the correlation between $\gamma$ and
$\gamma'$. Given a collection of edges $S$, denote by $V(S)$ the
collection of vertices which are endpoints for edges in $S$. The
next simple observation is of crucial importance for our proof.

\begin{lemma}\label{lem-vertex-edge}
For $1\leq \ell\leq n$ and $\gamma, \gamma'\in \Gamma_\ell$, Write
$S = E(\gamma) \cap E(\gamma')$. We have
$$|V(S)| = |S| + \theta(\gamma, \gamma')\,.$$
\end{lemma}
\begin{proof}
By definition, there exist no edge in $S$ that crosses different
$S$-components of $\gamma$, and there exists no vertex in $V(S)$
that is belonging to different $S$-components. Therefore, we can
analyze each $S$-component separately. In addition, it is obvious
that for each such $S$-component, the number of vertices is larger
than the number of edges by 1. Summing over all the $S$-components,
we complete the proof of the lemma.
\end{proof}
Given a path $\gamma'\in \Gamma_\ell$, we now partition $\Gamma_\ell$
based on its correlation with $\gamma$, i.e., based on the tuple
$(\theta(\gamma, \gamma'), |E(\gamma)\cap E(\gamma')|)$. Precisely,
for all integers $i\leq j$,  define
\begin{equation}\label{eq-def-A-i-j}
A_{i, j} \equiv A_{i, j}(\gamma) := \{\gamma'\in \Gamma_\ell:
\theta(\gamma, \gamma') = i, E(\gamma) \cap E(\gamma') =
j\}\,.\end{equation} It is now natural to control the cardinality of
$A_{i, j}$. One could prove more precise estimate on $A_{i,j}$, but
for our purpose, the following is sufficient.
\begin{lemma}\label{lem-counting}
For any $1\leq \ell\leq n$ and any $\gamma\in \Gamma_\ell$, we have
that for any non-negative integers $i \leq j$
$$|A_{i,j}(\gamma)| \leq \binom{\ell+1}{2i} \binom{n-i-j}{\ell+1-i-j} 2^i (\ell+1-j)! \leq \ell^{3i} n^{\ell+1-i-j}\,.$$
\end{lemma}
\begin{proof} In order to bound $|A_{i, j}|$, we
consider the following procedure to generate a path in $A_{i, j}$:
\begin{enumerate}
\setlength{\itemsep}{1pt}
  \setlength{\parskip}{0pt}
  \setlength{\parsep}{0pt}
\item Select $i$ vertex disjoint segments from $\gamma$ such that
the total number of edges is $j$ (and thus the total number of
vertices is $i+j$ by Lemma~\ref{lem-vertex-edge}).
\item Select $\ell+1 - i -j$ vertices from remaining $n - i- j$
vertices in the graph.
\item Choose a direction for each of the segment (two options for every segment). Then take each of the $i$ segments as an individual element and permute
these $i$ elements and $\ell+1-i-j$ vertices selected from step 2,
such that no more edges in $\gamma$ will be introduced when the
permutation is viewed as a path.
\end{enumerate}
Clearly, the cardinality of $A_{i, j}$ is bounded by the product of
the number of choices $N_k$ (for $k=1,2,3$) in each step. For Step
1, we see that the choice for the edges is complete determined by
the $2i$ endpoints selected from the path $\gamma$ and vice versa,
and thus $N_1 \leq \binom{\ell+1}{2i}$. For Step 2, we have $N_2 =
\binom{n-i-j}{\ell+1-i-j}$. For Step 3, it is obvious that $N_3 \leq
2^i (\ell+1-j)!$. Taking a product for $N_1$, $N_2$ and $N_3$, we
complete the proof of the lemma.
\end{proof}
We are now ready to give
\begin{proof}[Proof of Lemma~\ref{lem-second-moment}]
For $\ell = c (\log n)^3$ and $\gamma\in \Gamma_\ell$, we consider
$\Gamma_\ell$ as a union over $A_{i, j}$. For $i = 0$, the only
legitimate choice of $j$ is also 0, and in this case $E(\gamma')\cap
E(\gamma) = \emptyset$ for all $\gamma' \in A_{0, 0}$. Therefore,
the events $F_\gamma$ and $F(\gamma')$ are independent. Thus,
\begin{equation}\label{eq-independent-case}
\mbox{$\sum_{\gamma'\in A_{0, 0}}$} \P(F_\gamma \cap F_{\gamma'}) =
\mbox{$\sum_{\gamma'\in A_{0, 0}}$} \P(F_\gamma)\cdot\P(F_{\gamma'})
\leq \P(F_\gamma) \E N\,.\end{equation} Since $|E(\gamma)\cap
E(\gamma')| \geq \theta(\gamma, \gamma')$ always, we next consider
$1\leq i\leq j\leq \ell$. Take $\gamma'\in A_{i, j}$, and we wish to
bound the probability for the event $F_{\gamma'}$ conditioned on
$F_\gamma$. Write $S = E(\gamma) \cap E(\gamma')$ and $S' =
E(\gamma') \setminus S$. Conditioned on $F_\gamma$ (see the
definition of $F_\gamma$ in \eqref{eq-def-F}), we have that
$$\mbox{$\sum_{e\in S}$} X_e \geq |S| \lambda -1 - 2i\delta\log n\,.$$
Therefore, we obtain that
\begin{equation}\label{eq-condition-F}
\P(F_{\gamma'} \mid F_{\gamma}) \leq \P\left(\mbox{$\sum_{e\in S'}$}
X_e \leq \lambda |S'| + 1 + 2i\delta \log n \mid F_\gamma\right) =
\P\left(\mbox{$\sum_{e\in S'}$} X_e \leq \lambda |S'| + 1 +
2i\delta\log n\right)\,,\end{equation} where in the last inequality
we used independence of these exponential variables $X_e$. Recalling
that $|S'| = \ell-j$, we get from \eqref{eq-gamma-distribution} that
$$\P(F_{\gamma'} \mid F_{\gamma}) \leq O(1) n^{-(\ell - j)} n^{2i\delta}\,.$$
By Lemma~\ref{lem-counting}, we see that $A_{i, j} \leq
\ell^{3i} n^{\ell + 1 - i - j}$. A simple union bound then
yields that
$$\mbox{$\sum_{\gamma'\in A_{i, j}}$} \P(F_{\gamma'} \mid F_\gamma) \leq O(n) \ell^{3i} n^{-(1-2\delta) i }\,.$$
Summing over $1\leq i\leq j\leq \ell$, we then obtain that
$$\mbox{$\sum_{1\leq i\leq j\leq \ell} \sum_{\gamma' \in A_{i, j}}$} \P(F_{\gamma'} \mid F_\gamma)\leq O(1)\ell^5 n^{2\delta}\,.$$
Combined with \eqref{eq-independent-case}, it completes the proof of
the lemma.
\end{proof}

We next conclude this subsection with the proof for the lower bound
on $L(n, \lambda)$.

\begin{proof}[Proof of Theorem~\ref{thm-critical}: lower bound] Set
$\delta = \tfrac{1}{10}\min(1, 1/C^\star)$ and $c = \delta^3$. Since
$|\Gamma_\ell| = (1+o(1)) n^{\ell+1}$, the estimate
\eqref{eq-prob-F} then gives that for sufficiently enough $n$
$$\E N \geq n^{4/5}\,.$$
Meanwhile, by Lemma~\ref{lem-second-moment}, we have that
$$\E N^2 \leq \E N (\E N + n^{1/5}) = (1+o(1)) (\E N)^2\,.$$
At this point, a simple application of Chebyshev's inequality gives
that $\P(N>0)\to 1$ as $n\to \infty$, completing the proof for the
lower bound.
\end{proof}

\section{The existence of second transition point}
Throughout this section, we let $\lambda = \mathrm{e}^{-1} +
\epsilon $ and assume that $\epsilon \geq \beta (\log n)^{-2}$ for
an absolute large constant $\beta \geq 10$ to be specified later.
The goal of this subsection is to demonstrate the existence of
another phase transition at a point in $[\mathrm{e}^{-1} + \alpha
(\log n)^{-2}, \mathrm{e}^{-1} + \beta (\log n)^{-2}]$. To this end,
we prove Theorem~\ref{thm-supcritical} in this section, whose main
content is a polynomial lower bound on $L(n, \lambda)$ when $\lambda
- \mathrm{e}^{-1} \geq \beta (\log n)^{-2}$.

The upper bound for Theorem~\ref{thm-supcritical} follows from a
straightforward first moment computation, as quickly incorporated in
what follows.
\begin{proof}[Proof of Theorem~\ref{thm-supcritical}: upper bound]
Consider $\ell \geq 6\epsilon n$. For any $\gamma\in \Gamma_{\ell}$,
we have from \eqref{eq-gamma-distribution} that
$$\P(X(\gamma) \leq \lambda \ell) = O(1) ((\mathrm{e}^{-1}+\epsilon)\ell)^{\ell-1} \tfrac{1}{n^{\ell} (\ell-1)!} = O(1/\sqrt{\ell}) \mathrm{e}^{\epsilon \mathrm{e} \ell} \,.$$
In addition, we get that $|\Gamma_\ell| = \prod_{i=0}^{\ell}(n-i)
\leq n^{\ell+1} \mathrm{e}^{-\frac{\ell^2}{2n}}$. A simple union
bound over $\Gamma_\ell$ and $6\epsilon n\leq \ell\leq 12\epsilon n$
then gives that
$$\P\left(\exists 6\epsilon n\leq \ell\leq 12\epsilon n, \gamma\in \Gamma_{\ell}: X(\gamma) \leq \lambda \ell\right)\to 0\,.$$
Combined with Claim~\ref{claim-path-monotone}, this gives the upper
bound.
\end{proof}

\subsection{Deviation of typical light path:
revisited}\label{sec:deviation-revisited}

In view of the lower tail of the deviation as in
Lemma~\ref{lem-deviation}, it is obvious that when the length of the
path gets large, the tail gets extremely small and thus needs to be
tracked down carefully. In particular, we need an estimate for the
lower tail of the deviation given the values of some of the
variables along the path. We handle this delicate issue in this
subsection.
\begin{lemma}\label{lem-p-s}
Let $Z_i$ be i.i.d.\ exponential variables for $i\in \N$. Let
$1/4\leq \rho\leq 1$ and $M_n$ be defined as in \eqref{eq-def-M}.
Write for all $s\in \N$
\begin{equation}\label{eq-def-p-s} p_s = \P\left(M_s \leq r \mid
\mbox{$\sum_{i=1}^s$} Z_i = \rho s\right)\,.\end{equation} Then for
$j, k\in \N$, we have
$$p_{j+k} \geq \frac{1}{10^8 r \sqrt{j\wedge k}} p_j p_k\,.$$
\end{lemma}
\begin{proof}
Assume that $j\leq k$. The proof follows from a natural idea:
conditioning on the partial sum of the first $j$ variables. Given
that this partial sum is close to the expectation within a window of
size 1 (which occurs with probability $1/\sqrt{j\wedge k}$ as
$j\wedge k$ is the variance for this partial sum), the two segments
are independent and the probability for each of them to have
deviation smaller than $r$ is very close to $p_j$ and $p_k$. Noting
that the probability for the whole sequence to have deviation
smaller than $r$ is larger than the product of the three
aforementioned probabilities, we can then complete the argument. In
what follows, we carry out the technical details.

Denote by $\Omega_\delta = \{(z_i): \sum_{i=1}^{j+k} z_i = \rho
(j+k), \sum_{i=1}^j z_i=\rho j + \delta \rho\}$, and by $\Omega =
\cup_{0\leq \delta\leq 1/2}\Omega_\delta$. By
Claim~\ref{claim-exponential}, we see that
$$\P((Z_i)\in \Omega) \geq \tfrac{1}{2\cdot 10^{6}\sqrt{j}}\,.$$
Just for the technical reason (which will be clear later), we
partition $\Omega_0$ ($\Omega_0$ is defined as the aforementioned $\Omega_\delta$ with $\delta = 0$) into a union of sets $\Omega_{0, \tau}$ such
that $(z_i)\in \Omega_{0, \tau}$ if and only if $\tau = \min\{i\geq
j+1: z_i \geq \delta \rho\}$. Let $\Xi \in \R^{j+k}$ be such that
for all $(z_i)\in \Xi$
\begin{equation}\label{eq-condition-z}|\mbox{$\sum_{i=1}^s$} z_i - \rho s| \leq r \mbox{ for all }
1\leq s\leq j+k\,.\end{equation} Since sequences in $\Xi$ has small
deviation, we see that $\Xi \subseteq \cup_{\tau=1}^{2r} \Omega_{0,
\tau}$. Choose $\tau^*$ such that
$$\P\left((Z_i) \in \Xi \cap \Omega_{0, \tau^*}\right) = \max_{1\leq \tau\leq 2r} \P\left((Z_i) \in \Xi \cap \Omega_{0, \tau}\right) \geq \tfrac{1}{2r} \P((Z_i)\in \Xi\cap \Omega_0)\,.$$
Next, we show that for all $0\leq \delta\leq 1/2$, we have
\begin{equation}\label{eq-perturb}
\P((Z_i) \in \Xi \mid (Z_i)\in \Omega_\delta) \geq \tfrac{1}{10}
\P\left( (Z_i)\in \Xi \cap \Omega_{0, \tau^*}\mid (Z_i)\in
\Omega_0\right)\,.\end{equation} For $(z_i)\in \Omega_{0, \tau^*}
\cap \Xi$, we map $(z_i)$ to $(z'_i)$ by let $(z'_i)\in \Omega_\delta$
be such that $z'_i = z_i$ for $i\neq j, \tau^*$ and $z'_j = z_j +
\delta \rho$ and $z'_{\tau^*} = z_{\tau^*} - \delta \rho$ (the
assumption that $z_{\tau^*} \geq \delta \rho$ guarantees that
$z_{j+1} \geq 0$). Since $0\leq \delta\leq 1/2\leq r$, it is clear
that the sequence $(z'_i)$ also satisfies \eqref{eq-condition-z}.
Also, we see that the determinant of the Jacobian matrix of this
mapping is 1. It remains to compare the densities for $(Z_i)$ at
$(z_i)$ and $(z'_i)$ given $(Z_i)\in \Omega_0$ and $(Z_i)\in
\Omega_\delta$, respectively. It is obvious and straightforward to
check that the ratio of these two densities are within a constant
factor, say, $10$. This yields that
$$p_{j+k} \geq \tfrac{1}{10} \P((Z_i) \in \Omega)\P\left( (Z_i)\in \Xi\cap \Omega_{0, \tau^*} \mid \Omega_0\right) \geq \tfrac{1}{10^7 \cdot 2r \sqrt{j}} p_j p_k\,,$$
 where the last inequality we used conditional independence given
 $\Omega_0$. Altogether, this completes the proof of the lemma.
\end{proof}

\begin{lemma}\label{lem-deviation-conditioning}
Let $Z_i$ be i.i.d.\ exponential variables for $i\in \N$. Consider
$1\leq r\leq \sqrt{n}$ and $1\leq a_1\leq b_1 \leq a_{2} \leq \ldots
\leq a_m\leq b_m \leq n$ such that $q = \sum_{i=1}^m(b_i-a_i +
1)\leq n- 10 r$. Let $1/4\leq \rho\leq 1$ and $M_n$ be defined as in
\eqref{eq-def-M}.  Then for all $z_j$ such that
$$\mbox{$\sum_{j=a_i}^{b_i}$} z_j - \rho (b_i-a_i+1) \leq 2r\,,$$
we have (write $A = \cup_{i=1}^m [a_i, b_i] \cap \N$ and use the
notation of $p_s$ as in \eqref{eq-def-p-s})
$$\P\left(M_n \leq r \mid \mbox{$\sum_{i=1}^n$}Z_i = \rho n, Z_j = z_j \forall j\in A\right) \leq O(r \sqrt{q \wedge {n-q}})p_n10^{100mr}\mathrm{e}^{C^\star q/r^2}\,,$$
where $C^\star$ is the absolute constant from
Lemma~\ref{lem-deviation}.
\end{lemma}

\begin{proof}
We first sketch the outline of the proof. Since we are conditioned
on the sum of $Z_i$, by \eqref{eq-identity-in-law} the mean of $Z_i$
is irrelevant. Just for convenience, some times we assume that $Z_i$
has mean $\rho$. We will consider a new sequence of i.i.d.\
exponential variables $(Z'_i)$, whose average will also be
conditioned to be $\rho$. In addition, the size $n'$ of the new
sequence is larger than the number of free variables (the variables
that are not conditioned to be a given value) in the old sequence
$(Z_i)$ by $10 mr$. For each segment $[a_i, b_i]$, we force a
segment of size $10r$ in $(Z'_i)$ such that its partial sum (biased
by $\rho$ for each variable) grows almost linearly with end points
being $0$ and $\sum_{j\in [a_i, b_i]} (z_j - \rho)$ (say the
probability cost is $p$). Given this linear interpolation, the free
variables in $(Z_i)$ and $(Z'_i)$ will have almost the same
distribution and we couple them together. Furthermore, if the free
variables are such that the deviation in $(Z_i)$ is less than $r$
(say this occurs with probability $p'$), so should it be in $(Z'_i)$
by our construction. But we know that the probability for $(Z'_i)$
to have deviation smaller than $r$ is $p_{n'}$. Therefore, we can
deduce the bound $p_{n'}\geq p\cdot p'$. The technical details are
carried out in what follows.

Let $(Z'_i)_{1\leq i\leq n'}$ be i.i.d.\ exponential variables where
$n' = n-q + 10mr$. We first define the mapping $\phi(\cdot)$ between
the coordinates of the original sequence $(Z_i)$ and our new
sequence $(Z'_i)$, by
$$\phi(t) = |i\leq t: i\not\in A| + 10 r|i: b_i\leq t|\,.$$
Note that $\phi(t)$ remains constant over $[a_i, b_i)$. The
intuition behind is that we replace each segment $[a_i, b_i]$ in the
original sequence by a segment of size $10r$. Write $s_i =
\sum_{j=a_i}^{b_i} z_j -\rho(b_i - a_i + 1)$, for $1\leq i\leq m$.
By definition, we have $|s_i| \leq 2r$. For $0\leq \delta_i \leq
1/20$, we define $\Omega_{(\delta_i)} \subseteq \R^{n'}$ such that
$(x_j)_{1\leq j\leq n'} \in \Omega_{(\delta_i)}$ if for all $1\leq
i\leq m$
\begin{align*}\forall 1\leq k< 10r: &-1/20\leq \mbox{$\sum_{j=1}^k$} x_{\phi(b_i) - 10r +j} - k(\rho
+ s_i/10 r)\leq 0 ,\\
& \mbox{$\sum_{j=1}^{10r}$} x_{\phi(b_i) - 10r +j} = 10r\rho + s_i -
\delta_i\,.
\end{align*}
The idea, as we outlined before, is that we use the added segment to
linearly interpolate the (biased) endpoints of the old segment while
we keep the total sum (biased by mean) to be very close. Denote by
$\Omega = \cup_{(\delta_i)\in [0, 1/20]^m}\Omega_{(\delta_i)}$. A
repeated application of Claim~\ref{claim-exponential} gives that
$$\P((Z'_j)_{1\leq j\leq n'} \in \Omega) \geq 10^{-80mr}\,.$$
Next, consider $\Xi\subseteq \R^n$ such that $(x_j)\in \Xi$ if,
$$\forall j\in A: x_j = z_j, \mbox{ and } \forall k\in [n]:|\mbox{$\sum_{j=1}^k$} x_k - \rho k| \leq r, \mbox{ and } \mbox{$\sum_{j=1}^n$} x_j = \rho n\,.$$
Just as a translation of our definition of $\Xi$, we have
\begin{equation}\label{eq-translation}\P\left(M_n \leq r \mid \mbox{$\sum_{j=1}^n$}Z_j = \rho n, Z_j = z_j
\forall j\in A\right)= \P\left((Z_j)_{1\leq j\leq n} \in \Xi \mid
\mbox{$\sum_{i=j}^n$}Z_j = \rho n, Z_i = z_i \forall i\in A
\right)\,.\end{equation} Now for each $(\delta_i)$ and
$(\omega_j)_{1\leq j\leq n'}\in \Omega_{(\delta_i)}$, we construct a
mapping $\psi_{(\delta_i), (\omega_j)}: \Xi \mapsto
\Omega_{(\delta_i)}$ such that it maps $(x_j)_{1\leq j\leq n}\in
\Xi$ to $(y_j)_{1\leq j\leq n'}$
$$\forall j\in A: y_j = \omega_j\,, \forall j\not\in A\cup\{b_1, \ldots, b_m\}: y_{\phi_{j}} = x_j, \mbox{ and } y_{\phi_{b_i}+1} = x_{b_i+1} + \delta_i\,.$$
Crucially, the density of $(Z_j)_{1\leq j\leq n}$  at $(x_j)_{1\leq
j\leq n}$ given that $\{Z_j = z_j \forall j\in A\}$ and
$\sum_{j\in[n]} Z_j = \rho n$, is comparable with the density of
$(Z'_j)_{1\leq j\leq n'}$ at any $(y_j)_{1\leq j\leq n'}$ given that
$(Z'_j)_{1\leq j\leq n'}\in \Omega_{(\delta_i)}$ and $y_j = \omega_j
\forall j\in A$ and $\sum_{j=1}^{n'} Z'_j = \rho n'$. Indeed, the
ratio between these two densities can be directed computed and
founded to be within a factor of $10^m$. In order to see this, note
that given these conditions, the two random vectors have the same
number of free variables and the sum of these free variables differ
by amount of order $m$.

Define $\Xi'\subseteq \R^{n'}$ such that if $(y_j)\in \Xi'$,
$$(y_j) \in \Omega, \forall k\in [n']: |\mbox{$\sum_{j=1}^k$} y_j - k \rho| \leq r, \mbox{ and } \mbox{$\sum_{j=1}^{n'}$} y_j = \rho n'\,.$$
By definition, we can verify that for every $(\omega_j)\in
\Omega_{(\delta_i)}$
$$\psi_{(\delta_i), (\omega_j)}(\Xi) \in \Xi' \mbox{ for all } (\delta_i) \in [0, 1/20]^m\,.$$
We can then finally conclude that
\begin{align*}\P&\left((Z'_j)_{1\leq j\leq n'} \in \Xi' \mid
\mbox{$\sum_{j=1}^{n'}$}Z'_j = \rho n'\right)
\\
&\geq \P\left((Z'_j)_{1\leq j\leq n'} \in \Omega)
\min_{(\delta_i)}\P((Z'_j)_{1\leq j\leq n'} \in \Xi' \mid
\mbox{$\sum_{j=1}^{n'}$}Z'_j = \rho n',
(Z'_j)_{1\leq j\leq n'} \in \Omega_{(\delta_i)}\right) \\
&\geq 10^{-8rm} 10^{-m} \P\left((Z_j)_{1\leq j\leq n}\in \Xi \mid
Z_j = z_j \forall j\not\in A, \mbox{$\sum_{j=1}^n$}Z_n = \rho
n\right)\,.\end{align*} Combined with \eqref{eq-translation}, it
follows that
$$\P\left(M_n \leq r \mid \mbox{$\sum_{i=1}^n$}Z_i = \rho n, Z_i = z_i \forall i\in A\right)\leq 10^{(8r+1)m}p_{n'}\,.$$
Now the desired estimates follows from Lemma~\ref{lem-p-s} and
Lemma~\ref{lem-deviation}.
\end{proof}

\subsection{Lower bound}
Throughout this subsection, fix  $\lambda = \mathrm{e}^{-1} + \beta
(\log n)^{-2}$ for a large absolute constant $\beta>0$ to be
specified. Write $\ell = n^{1/4}$. Let $\zeta>0$ be a small constant
to be specified later. For $\gamma \in \Gamma_\ell$, define
\begin{equation}\label{eq-def-G}
G_\gamma = \{\lambda \ell - 1 \leq X(\gamma)\leq \lambda \ell,
M(\gamma) \leq \zeta \log n\cdot X(\gamma)/\lambda\ell\}\,.
\end{equation}
\noindent {\bf Remark.} Note that given that the event $G(\gamma)$
occurs, we always have $M(\gamma) \leq \zeta \log n + 1$. The extra
seemingly funny factor of $X(\gamma)/\lambda \ell$ is not crucial in
the definition of $G_\gamma$. It is merely for the purpose to have
the following: (which will save us some tedious effort)
\begin{equation}\label{eq-property-funny} \P(M(\gamma) \leq \zeta
\log n\cdot X(\gamma)/\lambda\ell \mid X(\gamma) = z) \equiv
\mathrm{constant} \mbox{ for all } \lambda \ell - 1\leq z\leq
\lambda\ell\,.
\end{equation}
Property \eqref{eq-property-funny} follows from the fact that for
all $z>0$:
$$\{\tfrac{1}{z}(X_e)_{e\in \gamma} \mid X(\gamma) = z\} \stackrel{law}{=} \{(X_e)_{e\in \gamma} \mid X(\gamma) = 1\}\,,$$
which one can verify by definition of exponential variables and
\eqref{eq-gamma-distribution}.

By \eqref{eq-gamma-distribution} and Lemma~\ref{lem-deviation}, we
obtain that
\begin{equation}\label{eq-prob-G} \P(G_\gamma) \geq \tfrac{1}{100}
n^{-\ell} \ell^{-1/2} \exp\big((\mathrm{e}\beta -
C^\star/\zeta^2)\ell(\log n)^{-2}\big)\,,\end{equation} where
$C^\star$ is the absolute constant from Lemma~\ref{lem-deviation}.
Defining $N = \sum_{\gamma\in \Gamma_\ell} \one_{G_\gamma}$, we see
that the first moment of $N$ would be large if we select $c, \delta$
properly.  In particular, we have
\begin{equation}\label{eq-first-moment-super}
\E N = \P(G_\gamma) |\Gamma_\ell| = (1+o(1))\P(G_\gamma)n^{\ell+1}
\geq \tfrac{(1+o(1))n}{100} \ell^{-1/2} \exp\big((\mathrm{e}\beta -
C^\star/\zeta^2)\ell(\log n)^{-2}\big)\,.
\end{equation}

As in Section~\ref{sec:critical-lower}, the key issue is to bound
the second moment of $N$. We use the basic ideas in
Section~\ref{sec:critical-lower}, with more delicate analysis. One
of the main difficulties is that now the probability cost for the
truncation on the deviation is so large such that it has to be
tracked down throughout, requiring delicate estimates on the
deviations of light paths (as incorporated in
Section~\ref{sec:deviation-revisited}) as well as a careful
treatment when patching estimates together.
\begin{lemma}\label{lem-second-moment-super}
For any $\gamma\in \Gamma_\ell$ and $\gamma'\in A_{i, j}$ with
$1\leq i\leq j$, we have that
$$\P(G_{\gamma'} \mid G_{\gamma}) \leq \P(G_\gamma)
O(\sqrt{\ell/(\ell-j)})
 n^j n^{300\zeta i}\mathrm{e}^{(C^\star/\zeta^2 - \mathrm{e}\beta)
j (\log n)^{-2}} \,.$$
\end{lemma}
\begin{proof}
Write $S = E(\gamma) \cap E(\gamma')$ and $S' = E(\gamma') \setminus
S$. Conditioned on $G_\gamma$ (see the definition of $G_\gamma$ in
\eqref{eq-def-G}), we have that
$$\mbox{$\sum_{e\in S}$} X_e \geq \lambda |S| -1 - 2i\zeta \log n = \lambda j - 1 - 2i\zeta \log n\,.$$
Therefore, we obtain that
$$\P(G_{\gamma'} \mid G_{\gamma}) = \P\left(\lambda \ell - 1\leq X(\gamma')\leq \lambda \ell \mid G_{\gamma}\right) \cdot\P\left(M(\gamma') \leq
 \zeta \log n \cdot X(\gamma')/\lambda\ell \mid G_{\gamma}, \lambda \ell - 1\leq X(\gamma')\leq \lambda \ell\right)\,.$$
It is clear that \begin{align*}\P\left(\lambda \ell - 1\leq
X(\gamma')\leq \lambda \ell \mid G_{\gamma}\right) &\leq
\P\left(\mbox{$\sum_{e\in S'}$} X_e \leq \lambda |S'| + 1 +
  2i\zeta \log n \mid G_\gamma\right)\\
  & = \P\left(\mbox{$\sum_{e\in S'}$} X_e \leq \lambda |S'| + 1 + 2i\zeta \log
  n\right)\,.\end{align*}
Recalling \eqref{eq-gamma-distribution} and that $|S'| = \ell-j$, we
get that
$$\P\left(\lambda \ell - 1\leq X(\gamma')\leq \lambda\ell \mid G_{\gamma}\right) \leq O((\ell-j)^{-1/2}) n^{-(\ell - j)} \mathrm{e}^{\mathrm{e}\beta(\ell-j) (\log n)^{-2}} n^{2\zeta i}\,.$$
By Lemma~\ref{lem-deviation-conditioning} and property
\eqref{eq-property-funny}, we obtain that
\begin{align*}\P&\left(M(\gamma') \leq
 \zeta \log n \cdot X(\gamma')/\lambda\ell \mid G_{\gamma}, \ell \lambda - 1\leq X(\gamma')\right)\\
 & \leq \P\left(M(\gamma) \leq
 \zeta \log n  \cdot X(\gamma)/\lambda \ell \mid \lambda \ell - 1 \leq X(\gamma) \leq \lambda \ell\right) \sqrt{j \wedge (n-j)} n^{300 \zeta i} \mathrm{e}^{C^\star j(\zeta \log
 n)^{-2}}\,.\end{align*}
Note that \begin{align*}\P(G_\gamma) &= \P\left(\lambda \ell - 1
\leq X(\gamma) \leq \lambda \ell\right)\P\left(M(\gamma) \leq \zeta
\log n \cdot X(\gamma)/\lambda\ell \mid
\lambda \ell - 1 \leq X(\gamma) \leq \lambda \ell\right)\,,\\
\P\left(G_{\gamma'} \mid G_{\gamma}\right)& = \P\left(\lambda \ell -
1 \leq X(\gamma') \leq \lambda \ell \mid
G_{\gamma}\right)\P\left(M(\gamma') \leq \zeta \log n \cdot
X(\gamma')/\lambda \ell \mid G_\gamma, \lambda \ell - 1 \leq
X(\gamma') \leq \lambda \ell\right)\,.\end{align*} Combining the
last four displays together, we complete the proof of the lemma.
\end{proof}
We next conclude this subsection with the proof for the lower bound
on $L(n, \lambda)$.
\begin{proof}[Proof of Theorem~\ref{thm-critical}: lower bound]
Set \begin{equation}\label{eq-zeta-beta}\zeta = 1/10^4, \mbox{ and }
\beta =  10^8 \max(C^\star, 1)\,.\end{equation} By
\eqref{eq-first-moment-super}, we see that
$$\E N \geq n^{3/4}\,.$$
Next, we turn to bound the second moment of $N$. For $\gamma\in
\Gamma_\ell$, consider $\Gamma_\ell$ as a union of $A_{i,j}$ over
$0\leq i\leq j\leq \ell$. For all $\gamma' \in A_{0, 0}$, the events
$G_\gamma$ and $G(\gamma')$ are independent. Thus,
\begin{equation}\label{eq-independent-case-super}
\mbox{$\sum_{\gamma'\in A_{0, 0}}$} \P(G_\gamma \cap G_{\gamma'}) =
\mbox{$\sum_{\gamma'\in A_{0, 0}}$} \P(G_\gamma)\cdot\P(G_{\gamma'})
\leq \P(G_\gamma) \E N\,.\end{equation} We next consider $1\leq
i\leq j\leq \ell$. For $\gamma'\in A_{i, j}$,
Lemma~\ref{lem-second-moment-super} and \eqref{eq-zeta-beta} gives
that
$$\P(G_{\gamma'}\mid G_{\gamma}) \leq \P(G_{\gamma}) O(\sqrt{\ell/\ell-j}) n^{j} n^{i/10} \mathrm{e}^{-10 j (\log n)^{-2}}\,.$$
Combined with Lemma~\ref{lem-counting}, it follows that
$$\mbox{$\sum_{\gamma'\in A_{i, j}}$}\P(G_{\gamma'}\mid G_{\gamma}) \leq \P(G_\gamma)n^{\ell+1} n^{-i/8} \mathrm{e}^{-8 j (\log n)^{-2}} = \E N \cdot n^{-i/8} \mathrm{e}^{-8 j (\log n)^{-2}}\,.$$
Summing over $1\leq i\leq j\leq \ell$ and recalling
\eqref{eq-independent-case-super}, we obtain that
$$\E(N\mid G_\gamma) \leq (1+ O(n^{-1/10})) \E N\,,$$
and therefore
$$\E N^2 = \mbox{$\sum_{\gamma\in \Gamma_\ell}$} \P(G_\gamma) \E(N\mid G_\gamma) = (1+ O(n^{-1/10})) (\E N)^2\,.$$
At this point, a simple application of Chebyshev's inequality gives
that $\P(N>0)\to 1$ as $n\to \infty$, completing the proof for the
lower bound.
\end{proof}

\section{Smooth interpolation through near sub-critical regime}

In this section, we demonstrate a smooth interpolation from
sub-criticality to criticality by proving
Theorem~\ref{thm-subcritical}. The proof uses similar ideas as
within the critical window and is also simpler. We write a separate
proof in order to reduce distractions for the presentation of the
core ideas in the critical regime. The proof in this section will be
presented in a concise way. Throughout, we let $\lambda =
\mathrm{e}^{-1} - \epsilon \leq \mathrm{e}^{-1} - (\log n)^{-2}$.

To prove the upper bound, we see that for $\ell = \epsilon^{-1} \log
n$ (by Claim~\ref{claim-path-monotone}),
$$\P(L(n, \lambda) \geq \ell) \leq \P\left(\exists \ell\leq \ell'<2\ell, \gamma\in \Gamma_{\ell'}: X(\gamma) \leq \lambda \ell'\right)\leq \mbox{$\sum_{\ell\leq \ell'<2\ell}$} n^{\ell'+1} n^{-\ell'} \mathrm{e}^{-\mathrm{e}\epsilon \ell'} = o(1)\,.$$

For the lower bound, consider $\ell = \frac{\log n}{10^4(C^\star
\vee 1) \epsilon}$, and for $\gamma\in \Gamma_\ell$ define
$$H_\gamma = \{\lambda\ell-1\leq X(\gamma)\leq \lambda \ell, M(\gamma)\leq \log n/10\}\,.$$
Then, by \eqref{eq-gamma-distribution} and
Lemma~\ref{lem-deviation}, we have that
$$\E N \geq n^{\ell+1} \tfrac{1}{100\sqrt{\ell}} n^{-\ell} \mathrm{e}^{-\mathrm{e}\epsilon \ell - 100C^\star\ell/(\log n)^2}\geq \sqrt{n}\,.$$
It remains to control the second moment of $N$. Analogous to
derivation of \eqref{eq-condition-F}, we obtain that for $\gamma'\in
A_{i, j}(\gamma)$
$$\P(H_{\gamma'} \mid H_\gamma) = O(1)n^{-(\ell-j)} n^{i/5 }\,.$$
Thus, by Lemma~\ref{lem-counting}, we obtain that for all $1\leq
i\leq j$
$$\mbox{$\sum_{\gamma'\in A_{i, j}}$} \P(H_{\gamma'} \mid H_\gamma) = O(\ell^3 n)n^{-4i/5}\,.$$
Summing over $0\leq i\leq j \leq \ell$, we obtain that
$$\E N^2 = (\E N)^2 + \E N \cdot \ell^5 n n^{-4/5} = (1+o(1))(\E N)^2\,.$$
The lower bound follows immediately.

\subsection*{Acknowledgements}

We are grateful to David Aldous for helpful communications, and we thank an anonymous referee for helpful comments correcting numerous typos in an earlier manuscript.

\end{document}